\documentclass[12pt,twoside]{article}
\usepackage{amssymb}
\usepackage{amsmath}
\usepackage{amsfonts}
\usepackage{graphicx}
\usepackage{enumerate}
\usepackage{amsthm}
\usepackage{calc}
\usepackage{tikz-cd}
\usepackage{float}
\usepackage{newlfont}
\usetikzlibrary{calc}

\newcommand{\RN}[1]{%
  \textup{\uppercase\expandafter{\romannumeral#1}}%
}
\date{}

\begin{document}

\vspace*{1.5cm}

\centerline{}

\centerline {\Large{\bf Maximal non valuative domains}} 
\centerline{}

\centerline{\bf {Rahul Kumar\footnote{The author was supported by the SRF grant from UGC India, Sr.
No. 2061440976.} \& Atul
Gaur\footnote{The author was supported by the MATRICS grant from DST-SERB India, No. MTR/2018/000707.}}}

\centerline{Department of Mathematics}

\centerline{University of Delhi, Delhi, India.}

\centerline{E-Mail: rahulkmr977@gmail.com; gaursatul@gmail.com}

\centerline{}

\newtheorem{Theorem}{\quad Theorem}[section]

\newtheorem{Corollary}[Theorem]{\quad Corollary}

\newtheorem{Lemma}[Theorem]{\quad Lemma}

\newtheorem{Proposition}[Theorem]{\quad Proposition}

\theoremstyle{definition}

\newtheorem{Definition}[Theorem]{\quad Definition}

\newtheorem{Example}[Theorem]{\quad Example}

\newtheorem{Remark}[Theorem]{\quad Remark}

\begin{abstract}
The notion of maximal non valuative domain is introduced and characterized. An integral domain $R$ is called a maximal non valuative domain if $R$ is not a valuative domain but every proper overring of $R$ is a valuative domain. Maximal non valuative domains have at most four maximal ideals. Various properties of maximal non valuative domains are discussed. Conditions are given under which pseudo-valuation domains and maximal non pseudo-valuation domains are maximal non valuative domains.
\end{abstract}


\section{Introduction}
Our work is motivated by \cite{cahen}. An integral domain $R$ with the quotient field $\text{qf}(R)$ is said to be a valuative domain, see \cite{cahen}, if for each $x\in \text{qf}(R)$, either $R\subseteq R[x]$ or $R\subseteq R[x^{-1}]$ has no intermediate ring. In this paper, we introduce the concept of maximal non valuative domains. Let $R\subset T$ be an extension of integral domains. Then we say that $R$ is a maximal non valuative subring of $T$ if $R$ is not a valuative domain but each subring of $T$ containing $R$ properly is a valuative domain. Moreover, if $T = \text{qf}(R)$, then $R$ is said to be a maximal non valuative domain. In this paper, we discuss various properties of maximal non valuative domain and characterize the same in terms of B$\acute{\text{e}}$zout domain.
All rings considered below are integral domains. By an overring of $R$, we mean a subring of the quotient field of $R$ containing $R$. A ring with a unique maximal ideal is called a local ring. The symbol $\subseteq$ is used for inclusion and $\subset$ is used for proper inclusion. Throughout this paper, $\text{qf}(R)$ denotes the quotient field of an integral domain $R$, $R'$ denotes the integral closure of $R$ in $\text{qf}(R)$. For a ring $R$, $\text{dim}(R)$ denotes the Krull dimension of $R$.

In this paper, we show that if $R$ is a maximal non valuative subring of $T$, then $T$ is an overring of $R$, see Theorem \ref{r1}. If $R$ is a maximal non valuative domain, then $R'$ is a Pr\"ufer domain, see Corollary \ref{sa2}. Moreover, if $R$ is not integrally closed, then $R$ has at most three maximal ideals, the set of non-maximal prime ideals of $R$ is linearly ordered by inclusion, and there is at most one maximal ideals of $R$ that does not contain all non-maximal prime ideals of $R$, see Theorem \ref{tar}; and if $R$ is integrally closed, then $R$ has at least two and at most four maximal ideals, there are exactly two non-maximal prime ideals that are not comparable in case $R$ has exactly two maximal ideals, otherwise the set of non-maximal prime ideals of $R$ is linearly ordered by inclusion, and there are at most two maximal ideals of $R$ that do not contain all non-maximal prime ideals of $R$, see Proposition \ref{sa1}. We characterize integrally closed maximal non valuative domains in terms of B$\acute{\text{e}}$zout domains, see Theorem \ref{r7}. Finally, we characterize local non integrally closed maximal non valuative domain $R$. We also discuss the cases where either $R$ is a pseudo-valuation domain or a maximal non pseudo-valuation subring of $R'$. 

For any ring $R$, $\text{Spec}(R)$ denotes the set of all prime ideals of $R$; $\text{Max}(R)$ denotes the set of all maximal ideals of $R$. As usual, $|X|$ denotes the cardinality of a set $X$.

\section{Results}

A ring extension $R\subseteq T$ is said to be residually algebraic if for any prime ideal $Q$ of $T$, $T/Q$ is algebraic over $R/(Q\cap R)$, see \cite{fontana,ayache1}. Moreover, if $R\subseteq S$ is residually algebraic, for any subring $S$ of $T$ containing $R$, then $(R,T)$ is said to be a residually algebraic pair, see \cite{ayache1}.

In our first theorem, we list some properties of an extension of integral domains in which every intermediate ring is a valuative domain. 

\begin{Theorem}\label{r1}
Let $R\subset T$ be a ring extension of integral domains. If each subring of $T$ properly containing $R$ is a valuative domain, then the following hold:
\begin{enumerate}
\item[(i)] $R\subset T$ is an algebraic extension.
\item[(ii)] $(R,T)$ is a residually algebraic pair.
\item[(iii)] If $R$ is not a field, then $T$ is an overring of $R$.
\end{enumerate}
\end{Theorem}

\begin{proof}
\begin{enumerate}
\item[(i)] If possible, suppose that $R\subset T$ is not an algebraic extension. Then there exists an element $t\in T\setminus R$ which is transcendental over $R$. Take $S = R[t^3, t^5]$. Then $S$ is a subring of $T$ containing $R$ properly. Therefore, $S$ is a valuative domain. Clearly, $t\in \text{qf}(S)$. Thus, either $S\subseteq S[t]$ or $S\subseteq S[t^{-1}]$ has no intermediate ring. Now, note that $S\subset R[t+t^{2}, t^3, t^5]\subset S[t] = R[t]$ and $S\subset R[t^{-1}+t^{2}, t^3, t^5]\subset S[t^{-1}] = R[t^{-1}, t^3, t^5]$, which is a contradiction. Hence, $R\subset T$ is an algebraic extension.

\item[(ii)] Let $S$ be a subring of $T$ properly containing $R$ and $Q$ be a prime ideal of $S$. If possible, suppose that $R/(Q\cap R) \subset S/Q$ is not an algebraic extension. Then there exists an element $\bar{t}\in S/Q$ that is not algebraic over $R/(Q\cap R)$. Take $S' = (R/Q\cap R)[\bar{t}]$. Then $S' = S''/(Q\cap S'')$ for some subring $S''$ of $S$ containing $R$ properly. Therefore, $S''$ is a valuative domain. Thus, by \cite[Theorem~2.2(i)]{cahen}, $S''$ has at most three maximal ideals and hence $S'$ has at most three maximal ideals, which is a contradiction.

\item[(iii)] 



Let $K = \text{qf}(R)$. If possible, suppose that $T\not\subseteq K$. Choose $t\in T\setminus K$. Then $t$ is algebraic over $R$, by part (i). Therefore, $\alpha = tr$ is integral over $R$ for some non zero $r\in R$. Clearly, $\alpha\notin K$. Let $n = [K(\alpha) : K]$. Then $\{1, \alpha, \alpha^2, \ldots, \alpha^{n-1}\}$ is a basis of $K(\alpha)$ over $K$. Let $\beta$ be any non zero, non unit of $R$. Take $S = R + \alpha\beta^{3} R + \alpha^{2}\beta^{3} R + \cdots + \alpha^{n-1}\beta^{3} R$. Then $S$ is a subring of $T$ properly containing $R$. Therefore, $S$ is a valuative domain. Note that $\text{qf}(S) = K(\alpha)$. Now, if $\alpha \beta^{-1}\in S$, then $\alpha \beta^{-1} = r_0 + \alpha\beta^{3} r_1 + \alpha^{2}\beta^{3} r_2 + \cdots + \alpha^{n-1}\beta^{3} r_{n-1}$ for some $r_0, r_1, \ldots, r_{n-1}\in R$. It follows that $\beta^{-1} = \beta^{3} r_1\in R$, a contradiction. Also, if $\alpha^{-1} \beta\in S$, then $\alpha^{-1} \beta = r_0 + \alpha\beta^{3} r_1 + \alpha^{2}\beta^{3} r_2 + \cdots + \alpha^{n-1}\beta^{3} r_{n-1}$ for some $r_0, r_1, \ldots, r_{n-1}\in R$. It follows that $\beta = r_0 \alpha + \alpha^{2}\beta^{3} r_1 + \alpha^{3}\beta^{3} r_2 + \cdots + \alpha^{n}\beta^{3} r_{n-1}$. Let $\alpha^{n} = \sum_{i = 0}^{n-1} \alpha^{i}x_{i}$ for some $x_0, x_1, \ldots, x_{n-1}\in R$. Consequently, we have $$\beta = r_0 \alpha + \alpha^{2}\beta^{3} r_1 + \alpha^{3}\beta^{3} r_2 + \cdots + \alpha^{n-1}\beta^{3} r_{n-2} + \beta^{3} r_{n-1} \Big(\sum_{i = 0}^{n-1} \alpha^{i}x_{i}\Big)$$
It follows that $\beta = \beta^{3} r_{n-1} x_0$, that is, $1 = \beta^{2} r_{n-1} x_0$, a contradiction. Thus, $\alpha\beta^{-1}, ~\alpha^{-1}\beta \in K(\alpha)\setminus S$. Since $S$ is a valuative domain, either $S\subset S[\alpha \beta^{-1}]$ or $S\subset S[\alpha^{-1} \beta]$ has no intermediate ring. Assume that $S\subset S[\alpha \beta^{-1}]$ has no intermediate ring. It follows that either $S = S[\alpha]$ or $S[\alpha] = S[\alpha \beta^{-1}]$. Now, if $S = S[\alpha]$, then $\alpha \in S$. Therefore, $\alpha = r_0 + \alpha\beta^{3} r_1 + \alpha^{2}\beta^{3} r_2 + \cdots + \alpha^{n-1}\beta^{3} r_{n-1}$ for some $r_0, r_1, \ldots, r_{n-1}\in R$. Consequently, we have $1 = \beta^{3} r_1$, which is a contradiction. Thus, we have $S[\alpha] = S[\alpha \beta^{-1}]$, that is, $\alpha \beta^{-1}\in S[\alpha]$. This gives $$\alpha \beta^{-1} = \Big(\sum_{i = 0}^{n-1} \alpha^{i}\beta^{3}r_{0i}\Big) + \Big(\sum_{i = 0}^{n-1} \alpha^{i}\beta^{3}r_{1i}\Big)\alpha + \cdots + \Big(\sum_{i = 0}^{n-1} \alpha^{i}\beta^{3}r_{(n-1)i}\Big)\alpha^{n-1}$$
Now, using $\alpha^{n} = \sum_{i = 0}^{n-1} \alpha^{i}x_{i}$, we get $\beta^{-1} = \beta^{3}r\in R$ for some non zero $r\in R$, a contradiction. Thus, we may assume that $S\subset S[\alpha^{-1} \beta]$ has no intermediate ring. It follows that either $S = S[\alpha^{-1} \beta^{2}]$ or $S[\alpha^{-1} \beta^{2}] = S[\alpha^{-1} \beta]$. If former holds, then $\alpha^{-1} \beta^{2}\in S$. Therefore, $\alpha^{-1} \beta^{2} = r_0 + \alpha\beta^{3} r_1 + \alpha^{2}\beta^{3} r_2 + \cdots + \alpha^{n-1}\beta^{3} r_{n-1}$ for some $r_0, r_1, \ldots, r_{n-1}\in R$. This gives $$\beta^{2} = (r_0 + \alpha\beta^{3} r_1 + \alpha^{2}\beta^{3} r_2 + \cdots + \alpha^{n-1}\beta^{3} r_{n-1})\alpha$$ Consequently, we have $\beta^{2} = \beta^{3}r$ for some non zero $r\in R$, that is, $1 = \beta r$, a contradiction. Finally, we assume that $S[\alpha^{-1} \beta^{2}] = S[\alpha^{-1} \beta]$, that is, $\alpha^{-1} \beta \in S[\alpha^{-1} \beta^{2}]$. This gives $$\alpha^{-1} \beta = \Big(\sum_{i = 0}^{n-1} \alpha^{i}\beta^{3}r_{0i}\Big) + \Big(\sum_{i = 0}^{n-1} \alpha^{i}\beta^{3}r_{1i}\Big)\alpha^{-1} \beta^{2} + \cdots + \Big(\sum_{i = 0}^{n-1} \alpha^{i}\beta^{3}r_{mi}\Big)\alpha^{-m} \beta^{2m},$$ that is, $$\alpha^{m-1} \beta = \Big(\sum_{i = 0}^{n-1} \alpha^{i}\beta^{3}r_{0i}\Big)\alpha^{m} + \Big(\sum_{i = 0}^{n-1} \alpha^{i}\beta^{3}r_{1i}\Big)\alpha^{m-1} \beta^{2} + \cdots + \Big(\sum_{i = 0}^{n-1} \alpha^{i}\beta^{3}r_{mi}\Big) \beta^{2m}$$ 
On comparing the coefficient of $\alpha^{m-1}$, we conclude that $\beta$ is a multiple of $\beta^{3}$, which is again a contradiction. Therefore, $T \subseteq K$. \qedhere


\end{enumerate}
\end{proof}

We now define the maximal non valuative subrings of an integral domain formally. 

\begin{Definition}
Let $R$ be a proper subring of an integral domain $T$. Then $R$ is said to be a maximal non valuative subring of $T$ if $R$ is not a valuative domain but every subring of $T$ properly containing $R$ is a valuative domain. A domain $R$ is said to be a maximal non valuative domain if $R$ is a maximal non valuative subring of its quotient field $\text{qf}(R)$. 
\end{Definition}

The next corollary shows that the integral closure of maximal non valuative domain is a Pr\"ufer domain.

\begin{Corollary}\label{sa2}
Let $R$ be a maximal non valuative domain. Then $R'$ is a Pr\"ufer domain. Moreover, if $R'$ is local, then $R'$ is a valuation domain. 
\end{Corollary}

\begin{proof}
Note that $(R, \text{qf}(R))$ is a residually algebraic pair, by Theorem $\ref{r1}$. Now, the result follows from \cite[Corollary~2.8]{ayache1}.\qedhere
\end{proof}


An integral domain $R$ is called an i-domain if for each overring $T$ of $R$, the canonical contraction map $\text{Spec}(T) \rightarrow \text{Spec}(R)$ is injective, see \cite{papick}. The next corollary is a direct consequence of Corollary \ref{sa2} and \cite[Corollary~2.15]{papick}.

\begin{Corollary}\label{a2}
Let $R$ be a maximal non valuative domain. If $R'$ is local, then $R$ is an i-domain.
\end{Corollary}

Now, in the next proposition we discuss the impact of localization on maximal non valuative subrings of a domain. 

\begin{Proposition}\label{sa3}
Let $R$ be a maximal non valuative subring of an integral domain $T$ and $N$ be a multiplicatively closed subset of $R$. Then either $N^{-1}R$ is a valuative domain or $N^{-1}R$ is a maximal non valuative subring of $N^{-1}T$.
\end{Proposition}

\begin{proof}
If $N^{-1}R$ is valuative, then we are done. Now, assume that $N^{-1}R$ is not valuative. Let $S'$ be a subring of $N^{-1}T$ containing $N^{-1}R$ properly. Then $S' = N^{-1}S$, for some subring $S$ of $T$ properly containing $R$. Now, by \cite[Proposition~2.4]{cahen}, $S'$ is valuative as $S$ is valuative by assumption. Thus, $N^{-1}R$ is a maximal non valuative subring of $N^{-1}T$.\qedhere
\end{proof}

As a consequence of Proposition \ref{sa3}, the requirement of ring to be local in Corollary \ref{a2} can be dropped. 
 
\begin{Corollary}\label{r5}
Let $R$ be a maximal non valuative domain. If $R$ is integrally closed, then $R$ is an i-domain.
\end{Corollary}

\begin{proof}
As i-domain is a local property, it is enough to show that $R_P$ is an i-domain for all prime ideals $P$ of $R$. Let $P$ be a prime ideal of $R$. Then by Proposition \ref{sa3}, either $R_P$ is a valuative domain or $R_P$ is a maximal non valuative domain. If $R_P$ is a valuative domain, then $R_P$ is an i-domain, by \cite[Corollary~3.3]{cahen}. Otherwise $R_P$ is a maximal non valuative domain. As $R_P$ is a local integrally closed domain, the result now follows from Corollary \ref{a2}.\qedhere
\end{proof}

In the next theorem, we list some properties of maximal non valuative domains which are not integrally closed. 

\begin{Theorem}\label{tar}
Let $R$ be a maximal non valuative domain. If $R$ is not integrally closed, then the following statements hold:
\begin{enumerate}
\item[(i)] $|\text{Max}(R)|\leq 3$.
\item[(ii)] The set of non-maximal prime ideals of $R$ is linearly ordered by inclusion.
\item[(iii)] There is at most one maximal ideal of $R$ that does not contain all non-maximal prime ideals of $R$.
\end{enumerate}
\end{Theorem}

\begin{proof}
Note that $R$ is a maximal non valuative subring of $R'$. In particular, $R'$ is a valuative domain.  Consequently, the statements (i), (ii), and (iii) hold for $R'$, by \cite[Theorem~2.2]{cahen}. As $R\subset R'$ is an integral extension of domains, we conclude that the statements (i), (ii), and (iii) hold for $R$ as well. \qedhere
\end{proof}



We now present some properties of maximal non valuative domains which are integrally closed. 

\begin{Proposition}\label{sa1}
Let $R$ be a maximal non valuative domain. If $R$ is integrally closed, then the following statements hold:
\begin{enumerate}
\item[(i)] $2\leq |\text{Max}(R)|\leq 4$.
\item[(ii)] If $|\text{Max}(R)| = 2$, then there are exactly two non-maximal prime ideals of $R$ that are not comparable. Otherwise, the set of non-maximal prime ideals of $R$ is linearly ordered by inclusion.
\item[(iii)] There are at most two maximal ideals of $R$ that do not contain all non-maximal prime ideals of $R$.
\end{enumerate}
\end{Proposition}

\begin{proof} 

\begin{enumerate}
\item[(i)] Note that if $R$ has more than four maximal ideals, then $R$ has a proper overring with four maximal ideals, which is a contradiction by \cite[Theorem~2.2(i)]{cahen}. It follows that $|\text{Max}(R)|\leq 4$. Now, if $R$ is local, then $R$ is a valuation domain, by Corollary \ref{sa2}, a contradiction.



\item[(ii)] First, assume that $M$ and $N$ are the only maximal ideals of $R$. Then $R_M$ and $R_N$ are valuative as $R$ is maximal non valuative. It follows that $R_M$ and $R_N$ are valuation domains, by \cite[Proposition~3.1]{cahen}. Thus, $R$ is a B$\acute{\text{e}}$zout domain with exactly two maximal ideals. Since $R$ is not a valuative domain, $M$ and $N$ do not contains each non-maximal prime ideal of $R$, by \cite[Theorem~3.7]{cahen}. Consequently, there are at least two non-maximal prime ideals of $R$ that are not comparable. Since $R_M$ and $R_N$ are valuation domains, there are exactly two non-maximal prime ideals of $R$ that are not comparable. 

Now, let $M_1, M_2$ and $M_3$ be any three maximal ideals of $R$. Then $R_{M_i}$ is a valuative domain for $i =1, 2, 3$. It follows that $R_{M_i}$ is a valuation domain for $i =1, 2, 3$, by \cite[Proposition~3.1]{cahen}. If possible, suppose that $P$ and $Q$ are any two incomparable non-maximal prime ideals of $R$. Without loss of generality, we may assume that $P\subset M_1$ but $P\not\subseteq M_2$ and $Q\subset M_2$ but $Q\not\subseteq M_1$. Since $R$ is maximal non valuative, $R_{M_1}\cap R_{M_2}$ is a valuative domain, which is a contradiction, by \cite[Theorem~3.7]{cahen}. Thus, the set of non-maximal prime ideals of $R$ is linearly ordered by inclusion. 

\item[(iii)] If $|\text{Max}(R)| = 2$, then nothing to prove. Now, assume that $R$ has exactly three maximal ideals, say $M_1, M_2$, and $M_3$. Let $P_i$ be a non-maximal prime ideal of $R$ such that $P_{i}\not\subset M_{i}$ for $i = 1, 2, 3$. Also by part (ii), we may assume that $P_1\subseteq P_2\subseteq P_3$. Then $P_3$ is a non-maximal prime ideal of $R$ that is not contained in any maximal ideal of $R$, a contradiction. Finally, assume that $R$ has exactly four maximal ideals, say $M_1, M_2, M_3$, and $M_4$. Let $P_i$ be a non-maximal prime ideal of $R$ such that $P_{i}\not\subset M_{i}$ for $i = 1, 2, 3$. Again, by part (ii), we may assume that $P_1\subseteq P_2\subseteq P_3$. Then $P_3 \subset M_4$. Note that $S = R_{M_1}\cap R_{M_2}\cap R_{M_4}$ is a valuative domain. Thus, by \cite[Theorem~2.2]{cahen}, at most one maximal ideal of $S$ does not contain each non-maximal prime ideal of $S$, a contradiction. Therefore, there are at most two maximal ideals of $R$ that do not contain all non-maximal prime ideals of $R$. \qedhere

\end{enumerate}       
\end{proof}

In the next theorem, we present a necessary and sufficient condition for an integrally closed domain to be a maximal non valuative domain. 

\begin{Theorem}\label{r7}
Let $R$ be an integrally closed domain. Then the following statements are equivalent:
\begin{enumerate}
\item[(1)] $R$ is a maximal non valuative domain. 
\item[(2)] Exactly one of the following holds:
\subitem(i) $R$ is a B$\acute{\text{e}}$zout domain with exactly two maximal ideals and exactly two non-maximal prime ideals of $R$ are not comparable. 
\subitem(ii) $R$ is a B$\acute{\text{e}}$zout domain with exactly three maximal ideals and exactly two maximal ideals of $R$ do not contain exactly one non-maximal prime ideal of $R$ whereas the third maximal ideal of $R$ contains all non-maximal prime ideal of $R$.
\subitem(iii) $R$ is a B$\acute{\text{e}}$zout domain with exactly four maximal ideals and at most one maximal ideal of $R$ does not contain all non-maximal prime ideals of $R$.

\end{enumerate}
\end{Theorem}

\begin{proof}
$(1)\Rightarrow (2)$ Note that $R$ is a Pr\"ufer domain, by Corollary \ref{sa2}. Also, by Proposition \ref{sa1}(i), we have $2\leq |\text{Max}(R)|\leq 4$. It follows that $R$ is a B$\acute{\text{e}}$zout domain. Now, assume that $|\text{Max}(R)| = 2$, then (i) follows from Proposition \ref{sa1}(ii). Also, if $|\text{Max}(R)| = 3$, then exactly two maximal ideals of $R$ do not contain at least one non-maximal prime ideal of $R$, by Proposition \ref{sa1}(iii) and \cite[Theorem~3.7]{cahen}. Now, assume that $M, N$ and $U$ are maximal ideals of $R$. If possible, assume that $P, Q$ are non-maximal prime ideals of $R$ such that only $U$ contains both of them. Since $R_U$ is a valuation domain, either $P\subset Q$ or $Q\subset P$. Without loss of generality, assume that $P \subset Q$. Then by \cite[Theorem~3.7]{cahen}, $R_M \cap R_N \cap R_Q$ is not a valuative domain, a contradiction as $R$ is maximal non valuative. Thus, $M$ and $N$ do not contain exactly one non-maximal prime ideal of $R$. Finally, assume that $|\text{Max}(R)| = 4$. Then by Proposition \ref{sa1}(iii), there are at most two maximal ideals of $R$ that do not contain all non-maximal prime ideals of $R$. If possible, assume that there are exactly two maximal ideals of $R$ that do not contain all non-maximal prime ideals of $R$. Let $M_1, M_2, M_3$ and $M_4$ be maximal ideals of $R$, where $M_1$ and $M_2$ do not contain all non-maximal prime ideals of $R$. Let $P_1, P_2$ be non-maximal prime ideals of $R$ such that $P_{1}\not\subset M_{1}$ and $P_{2}\not\subset M_{2}$. Moreover, by Proposition \ref{sa1}(ii), we may assume that $P_1\subseteq P_2$. Then $P_2\not \subset M_1$. Note that $P_2 \subset M_3$. Since $R$ is a maximal non valuative domain, $S = R_{M_1}\cap R_{M_2}\cap R_{M_3}$ is a valuative domain, which is a contradiction, by \cite[Theorem~3.7]{cahen}.


$(2)\Rightarrow (1)$ Suppose (i) holds. Then $R$ is not valuative, by \cite[Theorem~2.2]{cahen}. Now, let $M, N$ be maximal ideals of $R$ and $P, Q$ be incomparable non-maximal prime ideals of $R$. Since $R$ is a Pr\"ufer domain, $R_M, R_N$ are valuation domains. Thus, we may assume that $P\subset M$, $P\not\subset N$ and $Q\subset N$, $Q\not\subset M$. Now, by \cite[Corollary~3.9]{cahen}, it is enough to show that $R_M \cap R_U$ and $R_V \cap R_N$ are valuative domains for some arbitrary non-maximal prime ideals $U, V$ of $R$. First, we claim that $R_M \cap R_U$ is a valuative domain. If $U\subset M$, then we are done. Therefore, we may assume that $U\not \subset M$ and so $U\subset N$. It follows that $P\not\subseteq U$. Now if $U\subset P$, then again we are done. Otherwise $U = Q$, by assumption. Then by \cite[Theorem~3.7]{cahen}, $R_M\cap R_Q$ is a valuative domain. Similarly, we can prove that $R_V \cap R_N$ is a valuative domain.

Now, assume that (ii) holds. Then $R$ is not valuative, by \cite[Theorem~3.7]{cahen}. Let $S$ be a proper overring of $R$. Then $S = \cap_{P\in X} R_{P}$ for some subset $X$ of $\text{Spec}(R)$. Let $M, N$, and $U$ be maximal ideals of $R$ where $M$ and $N$ do not contain exactly one non-maximal prime ideal of $R$, whereas $U$ contains all non-maximal prime ideals of $R$. Now, the following cases arise:

Case (i): Let $M, N, U \in X$. Then $S = R$, a contradiction.

Case (ii): Let $U\in X$. Then $S = R_U$ is a valuation domain (and so is valuative) as $R$ is a B$\acute{\text{e}}$zout domain.

Case (iii): Let $M\in X$ but $N, U\notin X$. Then $R_M\cap R_U$ is a subring of $S$ contains $R$ properly. Note that $R_M\cap R_U$ is valuative, by \cite[Theorem~3.7]{cahen}. It follows that $S$ is valuative, by \cite[Corollary~3.9]{cahen}. Similarly, if $N\in X$ but $M, U\notin X$, then we are done. 

Case (iv): Let $M, N \in X$ but $U\notin X$. Then $S = R_M \cap R_N \cap (\cap_{P\in X_1} R_{P})$, where $X_1 = X\setminus \{M, N\}$. Without loss of generality, we may assume that $X_1$ does not contains any prime ideal that is either contained in $M$ or $N$. By \cite[Corollary~3.9]{cahen}, we may assume that $X_1$ is non empty. We claim that $X_1$ is a singleton set. If possible, suppose that $P, Q \in X_1$. By assumption, we may assume that $P\subset M, P\not\subset N$ and $Q \subset N, Q\not\subset M$. Since $P, Q \subset U$ and $R_U$ is a valuation domain, $P$ and $Q$ are comparable, which contradicts our assumption. Thus, we may assume that $X_1 = \{P\}$. Now, by \cite[Theorem~3.7]{cahen}, $S = R_M\cap R_N \cap R_P$ is a valuative domain. 

Case (v):  Let $M\not\in X, N\not\in X$, and $U\not\in X$. Then $S$ contains a valuation domain $R_U$ properly and hence $S$ is a valuation domain. 

Finally, assume that (iii) holds. Then again by \cite[Theorem~3.7]{cahen}, $R$ is not valuative. Let $S$ be a proper overring of $R$. Then $S = \cap_{P\in X} R_{P}$ for some subset $X$ of $\text{Spec}(R)$.
Now, the following cases arise:

Case (i): Let all four maximal ideals be in $X$. Then $R = S$, which is a contradiction.

Case (ii): Let there be more than one maximal ideals in $X$. Then $S$ is a B$\acute{\text{e}}$zout domain with at most three maximal ideals and at most one maximal ideal of $S$ does not contain all non-maximal prime ideals of $S$. Thus, $S$ is valuative, by \cite[Theorem~3.7]{cahen}. 

Case (iii): Let there be at most one maximal ideal in $X$. 

Subcase (i): Let all the maximal ideals of $R$ contain all non-maximal prime ideals of $R$. Then $S' = R_M$ is a subring of $S$ contains $R$ properly, where $M$ is a maximal ideal of $R$. Thus, $S'$ is a valuation domain and hence $S$ is a valuation domain. 

Subcase (ii): Let $N$ be the maximal ideal that does not contain all non-maximal prime ideals of $R$. If $N\notin X$, then again $S$ is a valuation domain. Now, assume that $N\in X$. Then take $S'' = R_M\cap R_N$, where $M$ is a maximal ideal of $R$ other than $N$. Note that $S''$ is a subring of $S$ contains $R$ properly. Now, $S''$ is valuative, by \cite[Theorem~3.7]{cahen}. Thus, by \cite[Corollary~3.9]{cahen}, $S$ is a valuative domain. Hence, $R$ is a maximal non valuative domain.\qedhere

\end{proof}
 
\begin{Corollary}\label{r9}
Let $R$ be a finite dimensional B$\acute{\text{e}}$zout domain. Assume that $\text{dim}(R) = n$. Then the following statements hold:
\begin{enumerate}
\item[(1)] If $R$ is a maximal non valuative domain, then $|\text{Spec}(R)| = n + |\text{Max}(R)|$, where $2 \leq|\text{Max}(R)| \leq 4$.

\item[(2)] $R$ is a maximal non valuative domain if and only if exactly one of the following holds:
\subitem(i) $|\text{Spec}(R)| = n + 2$, $R$ has exactly two maximal ideals with height $n$, and exactly two non-maximal prime ideals of $R$ are not comparable. 
\subitem(ii) $|\text{Spec}(R)| = n + 3$, $R$ has exactly three maximal ideals with exactly two maximal ideals of $R$ do not contain exactly one non-maximal prime ideal of $R$, and exactly one maximal ideal have height $n$.
\subitem(iii) $|\text{Spec}(R)| = n + 4$, $R$ has exactly four maximal ideals, and at least three of these maximal ideals have height $n$.
\end{enumerate}
\end{Corollary}

\begin{proof}
By Proposition \ref{sa1} and Theorem \ref{r7}, the result holds.  \qedhere
\end{proof}

A ring extension $R\subset T$ is said to be a minimal extension or $R$ is said to be a maximal subring of $R$, if there is no ring between $R$ and $T$, see $\cite{ferrand, modica}$. Moreover, $R\subset T$ is a pointwise minimal extension, if $R\subset R[x]$ is minimal for each $x\in T\setminus R$, see $\cite{cahen}$. Our next theorem gives a necessary and sufficient condition for a domain $R$ to be a maximal non valuative subring of $R'$ provided $R'$ is local.

\begin{Theorem}\label{r10}
Let $R$ be a ring such that $R'$ is local. Then $R$ is a maximal non valuative subring of $R'$ if and only if the following statements hold:
\begin{enumerate}
\item[(i)] $R'$ is a valuation domain.
\item[(ii)] $R\subset R'$ is not a pointwise minimal extension.
\item[(iii)] for each ring $S$ such that $R\subset S\subseteq R'$, we have either $S = R'$ or $S\subset R'$ is a pointwise minimal extension.
\end{enumerate}
\end{Theorem}

\begin{proof}
Let $R$ be a maximal non valuative subring of $R'$. Then $R'$ is a valuative domain and so is a valuation domain, by \cite[Proposition~3.1]{cahen}. Also, $R\subset R'$ is not a pointwise minimal extension, by \cite[Proposition~5.1]{cahen}. Let $S$ be a ring such that $R\subset S\subseteq R'$. Then $S$ is a valuative domain. Thus, by \cite[Proposition~5.1]{cahen}, either $S = R'$ or $S\subset R'$ is a pointwise minimal extension.

Conversely, assume that (i), (ii), and (iii) hold. Then $R$ is not a valuative domain by (ii) and \cite[Proposition~5.1]{cahen}. Also, every proper overring of $R$ contained in $R'$ is a valuative domain, by (i), (iii), and \cite[Proposition~5.1]{cahen}. 
\end{proof}

Recall from \cite{heds} that a domain $R$ is said to be a pseudo-valuation domain if for any prime ideal $P$ of $R$ and any $x, y$ in the quotient field of $R$ such that $xy \in P$, then either $x\in P$ or $y\in P$. Every PVD $R$ admits a canonically associated valuation overring $V$, in which every prime ideal of $R$ is also a prime ideal of $V$ and both $R$ and $V$ are local domains with the same maximal ideal, see \cite[Theorem~2.7]{heds}. In the next theorem, we give several equivalent conditions for a pseudo-valuation domain to be a maximal non valuative domain. 

\begin{Theorem}\label{r11}
Let $(R, M)$ be a pseudo-valuation domain, with canonically associated valuation overring $(V, M)$. Assume that $K := R/M$, $L := R'/M$, and
$F := V/M$. Then the following conditions are equivalent:
\begin{enumerate}
\item[(i)] $R$ is a maximal non valuative subring of $V;$
\item[(ii)] $R$ is a maximal non valuative domain;
\item[(iii)] $R' = V$, $R\subset R'$ is not a pointwise minimal extension, and for each ring $S$ such that $R\subset S\subseteq R'$, we have either $S = R'$ or $S\subset R'$ is a pointwise minimal extension;
\item[(iv)] $L = F$, $K\subset L$ is not a pointwise minimal extension, and for each ring $S$ such that $K\subset S\subseteq L$, we have either $S = L$ or $S\subset L$ is a pointwise minimal extension.
\end{enumerate}
\end{Theorem}

\begin{proof}
Note that (ii) $\Rightarrow$ (i) holds trivially by definition. For (i) $\Rightarrow$ (ii), assume that (i) holds. Let $T$ be a proper overring of $R$. Then either $T\subseteq V$ or $V\subset T$, by \cite[Lemma~1.3]{ben1}. If $T\subseteq V$, then $T$ is valuative. If $V\subset T$, then $T$ is a valuation domain. Thus, (ii) holds. Note that (i) $\Leftrightarrow$ (iii) follows from Theorem \ref{r10}. Finally it is easy to see (iii) $\Leftrightarrow$ (iv).\qedhere
\end{proof}

After pseudo-valuation domain the natural question is when maximal non pseudo-valuation ring is a maximal non valuative domain. This we address in the next theorem. Recall from $\cite{jarboui}$ that a maximal non pseudo-valuation subring of a domain $S$ is a proper subring $R$ of $S$ that is not a pseudo-valuation ring but each subring of $S$ properly containing $R$ is pseudo-valuation. Moreover, a ring $T$ is called the unique minimal overring of $R$ in $S$ if $R\subset T$ and any intermediate ring $A$ between $R$ and $S$ not equal to $R$ contains $T$, see $\cite{jarboui}$.

\begin{Theorem}
Let $R$ be a maximal non pseudo-valuation subring of $R'$. Then the following are equivalent:
\begin{enumerate}
\item[(i)] $R$ is a maximal non valuative subring of $R'$.
\item[(ii)] $R$ is not valuative, $R'$ is a valuation domain, and $R$ has a unique minimal overring $S$ in $R'$ that is valuative.
\end{enumerate}
\end{Theorem}

\begin{proof}
Let $R$ be a maximal non valuative subring of $R'$. Then $R'$ is valuation, by Theorem \ref{r10}. It follows that $R$ has a unique minimal overring, say $S$ in $R'$, by \cite[Theorem~6]{jarboui}. Note that $S$ is valuative as $R$ is a maximal non valuative subring of $R'$. 

Conversely, assume that (ii) holds. Let $T$ be a proper overring of $R$ in $R'$. Then $S\subseteq T$, by assumption. Since $R$ is a maximal non pseudo-valuation subring of $R'$, $S$ is a pseudo-valuation domain. It follows that $T$ is valuative, by \cite[Corollary~5.4]{cahen}. Thus, $R$ is a maximal non valuative subring of $R'$.
\end{proof} 

A proper overring $T$ of a domain $R$ is called the unique minimal overring of $R$ if any proper overring of $R$ contains $T$, see \cite{gilmer}.

\begin{Theorem}
Let $R$ be a maximal non pseudo-valuation subring of $\text{qf}(R)$. If $R$ is local, then the following are equivalent:
\begin{enumerate}
\item[(i)] $R$ is a maximal non valuative domain.
\item[(ii)] $R$ is not valuative and $R$ has a unique minimal overring $S$ that is a valuative pseudo-valuation domain with associated valuation overring $R'$.
\end{enumerate}
\end{Theorem} 

\begin{proof}
Let $R$ be a maximal non valuative domain. Then $R$ is not integrally closed, by Proposition \ref{sa1}. Now, (ii) follows from \cite[Theorem~4]{jarboui}. 

Conversely, assume that (ii) holds. Let $T$ be a proper overring of $R$. Then $S\subseteq T$, by assumption. It follows that $T$ is valuative, by \cite[Corollary~5.4]{cahen}. Thus, $R$ is a maximal non valuative domain.
\end{proof}

\end{document}